\documentclass[12pt]{amsart}
\usepackage{amscd, amssymb,latexsym,amsmath, amscd, amsmath}
\usepackage[all]{xy}
\usepackage{anysize}
\usepackage[sc]{mathpazo} 
\usepackage{color}
\usepackage{pst-node}
\marginsize{2.5cm}{2.5cm}{2.5cm}{2.5cm}
\newtheorem{theorem}{Theorem}[section] 
\newtheorem{lemma}[theorem]{Lemma}
\newtheorem{corollary}[theorem]{Corollary}
\newtheorem{proposition}[theorem]{Proposition}
\theoremstyle{definition}

\newtheorem{example}{Example}

\theoremstyle{remark}
\newtheorem{remark}{Remark}

\newcommand{\GL}{{\rm GL}}
\newcommand{\St}{{\rm St}}
\newcommand{\SL}{{\rm SL}}
\renewcommand{\1}{{\mathbf 1}}

\begin{document}

\title[Test vectors for finite periods and base change]
{Test vectors for finite periods and base change}
\author{U. K. Anandavardhanan and Nadir Matringe}

\address{Department of Mathematics, Indian Institute of Technology Bombay, Mumbai - 400076, India.}
\email{anand@math.iitb.ac.in}

\address{Laboratoire Math\'ematiques et Applications, Universit\'e de Poitiers, France.}
\email{matringe@math.univ-poitiers.fr}


\date{}

\begin{abstract}
Let $E/F$ be a quadratic extension of finite fields. By a result of Gow, an irreducible representation $\pi$ of $G = \GL_n(E)$ has at most one non-zero $H$-invariant vector, up to multiplication by scalars, when $H$ is $\GL_n(F)$ or ${\rm U}(n,E/F)$. If $\pi$ does have an $H$-invariant vector it is said to be $H$-distinguished. It is known, from the work of Gow, that $H$-distinction is characterized by base change from ${\rm U}(n,E/F)$, due to Kawanaka, when $H$ is $\GL_n(F)$ (resp. from $\GL_n(F)$, due to Shintani, when $H$ is ${\rm U}(n,E/F)$). Assuming $\pi$ is generic and $H$-distinguished, we give an explicit description of the $H$-invariant vector in terms of the Bessel function of $\pi$. Let $\psi$ be a non-degenerate character of $N_G/N_H$ and let $B_{\pi,\psi}$ be the (normalized) Bessel function of $\pi$ on the $\psi$-Whittaker model. For the $H$-average
\[W_{\pi,\psi} = \frac{1}{|H|} \sum_{h\in H} \pi(h) B_{\pi,\psi}\]  
of the Bessel function, we prove that 
\[W_{\pi,\psi}(I_n) = \frac{\dim \rho}{\dim \pi} \cdot \frac{|\GL_n(E)|}{|\GL_n(F)| |{\rm U}(n,E/F)|},\] 
where $\rho$ is the irreducible representation of ${\rm U}(n,E/F)$ (resp. $\GL_n(F)$) that base changes to $\pi$ when $H$ is $\GL_n(F)$ (resp. ${\rm U}(n,E/F)$).
As an application we classify the members of a generic $L$-packet of $\SL_n(E)$ that admit invariant vectors for $\SL_n(F)$. Finally we prove a $p$-adic analogue of our result for square-integrable representations in terms of formal degrees by employing the formal degree conjecture of Hiraga-Ichino-Ikeda \cite{hii08}.
\end{abstract}

\maketitle


\section{Introduction}\label{intro}

Let $F=\mathbb F_q$ be the finite field of $q$ elements for a prime power $q$. Let $E=\mathbb F_{q^2}$. For a connected reductive group $G$ defined over $F$, the question of which irreducible representations of $G(E)$ admit a non-trivial $G(F)$-invariant vector has been extensively studied \cite{gow84,pra99,lus00}. A representation of $G(E)$ which admits such a vector is said to be $G(F)$-distinguished. Of particular interest is the multiplicity free situation in which case there is a one dimensional space of $G(F)$-fixed vectors on an irreducible distinguished representation of $G(E)$. Gow proved that this is the case when $G=\GL(n)$ or $G={\rm U}(n)$ \cite{gow84}. 

Let $\sigma$ be the Frobenius automorphism given by $\sigma(x)=x^q$ for $x \in E$. We continue to denote by $\sigma$ the involution on $\GL_n(E)$ which is applying $\sigma$ entry wise. Define the involution $\tau$ on $\GL_n(E)$ by
\[g^\tau = J~ {^t}g^{-\sigma} J^{-1},\]
where $J$ is the longest Weyl element in $\GL_n(E)$ consisting of $1$'s on the anti-diagonal and $0$'s elsewhere. Note that $\GL_n(F)$ is the fixed points of $\sigma$ and we take ${\rm U}(n,E/F)$ to be subgroup of $\GL_n(E)$ consisting of the fixed points of $\tau$.

Let $\pi$ be an irreducible representation of $GL_n(E)$. If $\iota$ is an involution on $\GL_n(E)$, we define the representation $\pi^\iota$ by $\pi^\iota(g)=\pi(g^\iota)$. Assume $\pi^\iota \cong \pi$. Then there is a linear transformation, say $T_\iota$, unique up to a scalar multiple, on the space of $\pi$ such that
\[\pi(g) \circ T_\iota = T_\iota \circ \pi(g^\iota).\]
Let $\iota$ be either $\sigma$ or $\tau$.
It is known that there exists an irreducible representation $\rho$ of 
\[G_\iota = \{g \in \GL_n(E) \mid g^\iota = g\}\]
such that
\[{\rm Trace}~ [\pi(g)T_\iota] = {\rm Trace}~ [\rho(g g^\iota)],\]
for a suitable normalization of $T_\iota$. We say that the $\iota$-invariant representation $\pi$ of $\GL_n(E)$ is the base change lift of the representation $\rho$ of $G_\iota$ or that $\rho$ base changes to $\pi$. This is the work of Shintani when $\iota = \sigma$ \cite{shi76}, and Kawanaka when $\iota = \tau$ \cite{kaw77}. We remark that Kawanaka assumes the characteristic $p$ to be odd, so it is understood that we have the same assumption whenever we deal with base change from ${\rm U}(n,E/F)$. 

It follows from the work of Gow \cite{gow84} that distinction is characterized by base change. An irreducible representation $\pi$ of $\GL_n(E)$ is distinguished with respect to $\GL_n(F)$ if and only if it is a base change lift from ${\rm U}(n,E/F)$ and it is distinguished with respect to ${\rm U}(n,E/F)$ if and only if it a base change lift from $\GL_n(F)$. 

The main aim of this paper is to take the next step after Gow's result which is to describe the unique invariant vector explicitly in a given model for the representation. We assume that the representation is generic, that is to say that it has a Whittaker model, and we answer this question on the Whittaker model. We do this for the symmetric pairs $(G,H)$ considered by Gow, namely, $(\GL_n(E),\GL_n(F))$ and $(\GL_n(E),{\rm U}(n,E/F))$.

Recall that an irreducible representation $\pi$ of $\GL_n(E)$ contains a non-degenerate character $\psi$ of $N(E)$ at most once by a well-known result of Gelfand and Graev \cite{gg62}. The irreducible representation $\pi$ is said to be $\psi$-generic if it does contain such a $\psi$ and the unique realization of $\pi$ in the representation induced from $N(E)$ to $G(E)$ of $\psi$ is called the $\psi$-Whittaker model of $\pi$, denoted by $\mathcal W(\pi,\psi)$. An irreducible representation of $G(E)$ is generic with respect to one non-degenerate character of $N(E)$ if and only if it is generic with respect to any non-degenerate character of $N(E)$. Throughout this paper we fix the character $\psi$ that we work with; we choose $\psi$ so that it is trivial on restriction to the group $N_H$ of unipotent upper triangular matrices in $H$. This is equivalent to saying that $\psi^\tau = \psi$ (resp. $\psi^\sigma = \psi$) when $H=\GL_n(F)$ (resp. ${\rm U}(n,E/F)$). Now if $\pi^\iota \cong \pi$ then $\mathcal W(\pi,\psi) = \mathcal W(\pi^\iota,\psi)$, by the uniqueness of the Whittaker model, and the $\iota$-twisted intertwining operator $T_\iota$ can be realized on $\mathcal W(\pi,\psi)$ as
\[T_\iota (W) = W^\iota,\]
where $W^\iota(g) = W(g^\iota)$.

For an irreducible generic representation $\pi$ of $\GL_n(E)$, observe that there is a unique vector $B_\pi \in \mathcal W(\pi,\psi)$, up to multiplication by scalars, such that  
\[B_\pi(mgn) = \psi(mn)B_\pi(g),\]
for $g \in \GL_n(E)$, and $m,n \in N(E)$, by the multiplicity one result of Gelfand and Graev \cite{gg62}. This space is investigated by Gelfand in \cite{gel70}, and in particular such a vector, called the Bessel function of $\pi$, has an explicit description in terms of the character of $\pi$ \cite[Proposition 4.5]{gel70}. If $\chi_\pi$ denotes the character of $\pi$ then the Bessel function is given by
\[B_\pi(g) = \frac{1}{|N(E)|} \sum_{n \in N(E)} \psi^{-1}(n) \chi_\pi(gn).\] 
Note that $B_\pi(I_n) =1$. Thus, in the multiplicity free situation of $(\GL_n(E),N(E),\psi)$ considered by Gelfand-Graev \cite{gg62}, the unique $\psi$-invariant vector in the $\psi$-Whittaker model can be explicitly described and this is in terms of the character of $\pi$.

The main theorem of this paper is an explicit description of the unique vector in the multiplicity free situation of $(\GL_n(E),H,1)$, where $H$ is either $\GL_n(F)$ or ${\rm U}(n,E/F)$, considered by Gow \cite{gow84}. For an irreducible generic representation of $\GL_n(E)$, we do this on the Whittaker model in terms of the Bessel function of $\pi$. Note that given any Whittaker function $W$ in $\mathcal W(\pi,\psi)$, the sum of its $H$-translates, namely $\displaystyle{\sum_{h \in H}} \pi(h)W$, is obviously $H$-invariant, so the point is to choose a Whittaker vector $W$ so that this sum is non-zero. Equivalently, we are interested in the (obviously $H$-invariant) linear form on $\mathcal W(\pi,\psi)$ given by
\[W \mapsto \sum_{h \in H} W(h),\]     
and in finding a vector on which this linear form is non-vanishing. Such a vector is called a test vector for the linear form. We show that the Bessel function is a test vector for this linear form. In fact we derive a formula for the value of this linear form at the Bessel function. It is striking that our formula involves base change and thus not only that base change characterizes distinction \cite{shi76,kaw77,gow84}, but the relationship between distinction and base change is even deeper and it reflects at the level of a test vector! 

\begin{theorem}\label{thm-main}
Let $\pi$ be an irreducible generic representation of $\GL_n(E)$ which is distinguished with respect to $H$ which is either $\GL_n(F)$ or ${\rm U}(n,E/F)$. Let $H^\prime$ denote ${\rm U}(n,E/F)$ when $H= \GL_n(F)$, and $\GL_n(F)$ when $H={\rm U}(n,E/F)$. Let $\rho$ be the representation of $H^\prime$ that base changes to $\pi$. Let $\psi$ be a non-degenerate character of $N(E)/N_H$. Let $\mathcal W(\pi,\psi)$ be the $\psi$-Whittaker model of $\pi$, and let $B_\pi \in \mathcal W(\pi,\psi)$ be the Bessel function of $\pi$. Consider the $H$-invariant linear form on $\mathcal W(\pi,\psi)$ given by  
\[\lambda(W) = \frac{1}{|H|}\sum_{h \in H} W(h).\]
Then, $\lambda(W)$ is independent of the non-degenerate character $\psi$ of $N(E)/N_H$, and
\[\lambda(B_\pi) = \frac{\dim \rho}{\dim \pi} \cdot \frac{|\GL_n(E)|}{|\GL_n(F)| |{\rm U}(n,E/F)|}.\]
\end{theorem}

\begin{remark}\label{nov19}
The proof of the above identity (cf. \S \ref{pf-thm-main}) shows that an irreducible generic representation of $\GL_n(E)$ which is a base change from $H^\prime$ is $H$-distinguished, without appealing to the results of Gow \cite{gow84}. Recently, Yang has shown that the methods of this paper generalize to all connected reductive groups $G$; in particular representations of $G(E)$ which are generic for a distinguished non-degenerate character and in the image of the Shintani base change map from $G^{op}(F)$ (Prasad's opposition group, see \cite{pra15}) are automatically $G(F)$-distinguished \cite{yang19}.
\end{remark}

\begin{corollary}\label{cor-1}
The unique, up to multiplication by scalars, non-trivial $H$-invariant vector $W_\pi \in \mathcal W(\pi,\psi)$ is given by
\[W_\pi = \sum_{h \in H} \pi(h)B_\pi.\]
Moreover, if $\Lambda$ is the $\psi$-Whittaker functional on $\pi$ then $\Lambda (W_\pi) \neq 0$.
\end{corollary}

\begin{remark}\label{mao}
In \cite{mao01}, Mao has investigated the $H$-Bessel function associated to a $\psi$-generic $H$-distinguished representation of a group $G$ where $(G,H)$ is a multiplicity one symmetric pair. According to \cite[Proposition 3]{mao01}, there exists $\alpha_\pi \in G$ such that 
\[i_\pi(g) = \frac{1}{|N||H|}\sum_{n \in N}\sum_{h \in H}\psi^{-1}(n)\chi_\pi(\alpha_\pi h g n),\]
is non-zero and any $H$-Bessel function is a multiple of this function. It is easy to see that $\alpha_\pi \neq 1$ if $\psi$ is non-trivial on $N \cap H$. If $\psi$ is trivial on $N \cap H$, one consequence of Theorem \ref{thm-main} is that $\alpha_\pi$ can be chosen to be $1$ in the cases of $(G,H)$ considered in this paper. Indeed, with $\alpha_\pi = 1$, 
\[i_\pi(g) = \lambda(\pi(g)B_\pi),\]
and Theorem \ref{thm-main} evaluates $i_\pi(1)$ to be non-zero.
\end{remark}

The next theorem concerns the symmetric pair $(\GL_n(E),\GL_n(F))$. This theorem is a consequence of Theorem \ref{thm-main} as well as Proposition \ref{prop-gelfand}. To state the theorem we introduce the linear functional on $\mathcal W(\pi,\psi)$ given by
\[\ell(W) = \frac{1}{|\GL_n(F)|} \sum_{p \in P(F)} W(p),\]
where $P(F)$ denotes the mirabolic subgroup of $\GL_n(F)$. The linear form is obviously $P(F)$-invariant. Observe that it is $\GL_n(F)$-invariant precisely when $\pi$ is distinguished with respect to $\GL_n(F)$ and the dimension of the space of $P(F)$-fixed vectors of $\pi$ is one dimensional; such irreducible representations are said to be relatively cuspidal and they afford a nice characterization (cf. Corollary \ref{cor-P-multiplicity one}). Thus, if $\pi$ is a relatively cuspidal representation of $\GL_n(E)$ then both the linear forms $\lambda$ and $\ell$ on $\mathcal W(\pi,\psi)$ are $\GL_n(F)$-invariant and hence they differ by a scalar. We compute this scalar in the following theorem.

\begin{theorem}\label{thm-scalar}
Let $\pi$ be a relatively cuspidal representation of $\GL_n(E)$. Let $\rho$ be the representation of ${\rm U}(n,E/F)$ that base changes to $\pi$. Then,
\[\lambda = \frac{|\GL_n(E)/N(E)|}{|{\rm U}(n,E/F)/N(F)|} \cdot \frac{\dim \rho}{\dim \pi} \cdot \ell.\]
\end{theorem}

Though most of this paper concerns only with finite fields, a comparison with the corresponding picture for $p$-adic fields may be illuminating. For a quadratic extension $K/k$ of $p$-adic fields, multiplicity one is true only for $(\GL_n(K),\GL_n(k))$ and not for $(\GL_n(E),{\rm U}(n,K/k))$. For the purposes of comparison we restrict ourselves to the former symmetric pair in which case the unique invariant linear form can be explicitly realized on the Whittaker model for any irreducible generic distinguished representation and there exists an explicit test vector for this invariant form. For more details, we refer to \cite{am17}. Suffices to say here that the mirabolic subgroup $P(k)$ of $\GL_n(k)$ plays a crucial role in the $p$-adic setting whereas it is of not much use in the finite field case. The point is that the space of $P(F)$-invariants on an irreducible representation of $\GL_n(E)$ can be quite large in general, and in fact we precisely describe this space in \S \ref{bz}. Ultimately the finite field case is different due to semisimplicity of representations and this lack of non-trivial extensions is related to the lack of a
non-trivial absolute value.  

The $p$-adic case dealt with in \cite[\S 7]{am17} plays an important role in the study of distinction for the symmetric pair $(\SL_n(K),\SL_n(k))$ \cite{ap18}. Closely following the methods of \cite{ap18}, our main theorem can be used to characterize the members of a generic $L$-packet of $\SL_n(E)$ that admit $\SL_n(F)$-invariant vectors in terms of $\psi$-genericity. We refer to \S \ref{sln} for the precise result.

In \S \ref{hii}, we prove a $p$-adic analogue of Theorem \ref{thm-main} for square-integrable representations for the symmetric pair $(\GL_n(K),\GL_n(k))$, where $K/k$ is a quadratic extension of $p$-adic fields. Since we are dealing with square-integrable representations an exact analogue of the linear form $\lambda$ of Theorem \ref{thm-main} is known to exist and we show that its value at suitably chosen Whittaker functions of a square-integrable $\GL_n(k)$-distinguished representation $\pi$ is given by $d(\rho)/d(\pi)$, up to a non-zero constant not depending on the representations, where $\rho$ is the square-integrable representation of ${\rm U}(n,K/k)$ that base changes to $\pi$, and $d(\rho)$ (resp. $d(\pi)$) denotes the formal degree of the representation $\rho$ (resp. $\pi$) (cf. Remark \ref{essential}). 

In fact it is more illuminating to state the $p$-adic analogue of Theorem \ref{thm-scalar}. We do this now. Let $\pi$ be a square-integrable representation of $\GL_n(K)$ which is distinguished with respect to $\GL_n(k)$. The $p$-adic analogues of the linear forms of Theorem \ref{thm-scalar} do exist and are invariant under $\GL_n(k)$ \cite{fli88,akt04,kab04}. Thus, the two $\GL_n(k)$-invariant forms on $\mathcal W(\pi,\psi)$ are given by     
\[\lambda(W) = \int_{k^\times N(k) \backslash \GL_n(k)} W(h)dh,\] and
\[\ell(W) = \int_{N(k) \backslash P(k)} W(h)dh.\]
\begin{theorem}\label{thm-p}
Let $\pi$ be a square-integrable representation of $\GL_n(K)$ which is distinguished with respect to $\GL_n(k)$. Let $\rho$ be the square-integrable representation of ${\rm U}(n,K/k)$ that base changes to $\pi$, stably or unstably depending on the parity of $n$. For $W \in \mathcal W(\pi,\psi)$, we have
\[\lambda = c \cdot \frac{d(\rho)}{d(\pi)} \cdot \ell,\]
where $c$ is a positive constant that does not depend on the representations $\rho$ and $\pi$.
\end{theorem}

\begin{remark}
We may mention here that the identity in Theorem \ref{thm-p} is highly non-trivial as it involves a root number (which turns out to be $1$). Indeed what we show in \S \ref{hii} is that
\[\lambda = c \cdot \epsilon(1/2,\pi,r,\psi) \cdot \frac{d(\rho)}{d(\pi)} \cdot \ell,\]
where $\epsilon(1/2,\pi,r,\psi)$ is the Asai root number of $\pi$. It was expected that this root number is $1$ under our assumptions and this is now known \cite{ana08,akmss18,bp18a,shank18}.
\end{remark}

\begin{remark}
The constant $c>0$ is intimately related to the formal degree conjecture of Hiraga, Ichino, and Ikeda \cite{hii08} for the groups $\GL_n(K)$ and ${\rm U}(n,K/k)$. Note that in its finite analogue, namely Theorem \ref{thm-scalar}, the constant $c$ is made explicit and has an especially nice form. Consequently Theorem \ref{thm-scalar} is indicative of a finite field analogue of the the formal degree conjecture of \cite{hii08} (cf. Remark \ref{fd}). 
\end{remark}

\section{Preliminaries}\label{prelim}

Let $F = \mathbb F_q$ and $E= \mathbb F_{q^2}$. Define $\sigma$ on $E$ by $\sigma (x) = x^q$. We define the involution $\sigma$ on $G=\GL_n(E)$ by
\[\sigma ((a_{ij})) = (a_{ij}^\sigma) \]
and the involution $\tau$ on $\GL_n(E)$ by
\[g^\tau = J~ {^t}g^{-\sigma} J^{-1},\]
where $J$ is the longest Weyl element in $\GL_n(E)$ consisting of $1$'s on the anti-diagonal and $0$'s elsewhere. For an involution $\iota$ of $\GL_n(E)$, let
\[G_\iota = \{g \in \GL_n(E) \mid g^\iota = g\}.\]
We have $G_\sigma = \GL_n(F)$, and we define ${\rm U}(n,E/F)$ to be $G_\tau$.

Let $\pi$ be an irreducible representation of $\GL_n(E)$. The representation $\pi$ is said to be distinguished with respect to a subgroup $H$ if it has an $H$-invariant vector. We know that an irreducible representation $\pi$ of $\GL_n(E)$ has at most a one dimensional space of $G_\iota$-invariant vectors where $\iota$ is either $\sigma$ or $\tau$ \cite[Theorem 2.1, Theorem 3.6]{gow84}. Define the representation $\pi^\iota$ of $\GL_n(E)$ by $\pi^\iota(g) = \pi(g^\iota)$. It is known that $\pi$ is distinguished with respect to $H_\sigma$ if and only if $\pi \cong \pi^\tau$ and it is distinguished with respect to $H_\tau$ if and only if $\pi \cong \pi^\sigma$ \cite[Theorem 2.4, Theorem 3.6]{gow84}.

Let $\pi \cong \pi^\iota$. There exists a linear transformation $T_\iota$, which is unique up to a scalar multiple, on the space of $\pi$ such that
\[\pi(g) \circ T_\iota = T_\iota \circ \pi(g^\iota).\]
Let $\iota$ be either $\sigma$ or $\tau$. It is known that there exists an irreducible representation $\rho$ of $G_\iota$
such that
\[{\rm Trace}~ [\pi(g)T_\iota] = \chi_\rho(g g^\iota),\]
for a suitable normalization of $T_\iota$. Here $\chi_\rho$ denotes the character of $\rho$. Note that $gg^\iota$ is conjugate to an element of $G_\iota$, by Lang's theorem, and therefore $ \chi_\rho(g g^\iota)$ makes sense. The $\iota$-invariant representation $\pi$ of $\GL_n(E)$ is said to be the base change lift of the representation $\rho$ of $G_\iota$. When $\iota=\sigma$, base change is established by Shintani \cite[Theorem 1]{shi76}, and  this is done by Kawanaka when $\iota = \tau$ \cite[Theorem 4.1 (b)]{kaw77}. We note that Kawanaka assumes that the characteristic of $F$ is odd. 

Let $N=N(E)$ be the subgroup of $\GL_n(E)$ consisting of unipotent upper triangular matrices. If $\psi$ is a non-degenerate character of $N$ then we know that the induced representation ${\rm Ind}_N^G \psi$ is multiplicity free \cite{gg62}. This is called the Gelfand-Graev representation. An irreducible representation $\pi$ of $\GL_n(E)$ is said to be generic if it appears in the Gelfand-Graev representation. The $\psi$-Whittaker model of $\pi$ is given by the vector space
\[\mathcal W(\pi,\psi) = \{W: G \rightarrow \mathbb C \mid W(ng) = \psi(n) W(g) \mbox{~for~} g \in G, n \in N\},\]
on which $\pi$ acts as
\[\pi(g^\prime)W(g)=W(gg^\prime).\] 
It follows that if $\pi$ is generic then the space
\[\{B \in \mathcal W(\pi,\psi) \mid \pi(n)B = \psi(n)B\}\]
is one dimensional. By \cite[Proposition 4.5]{gel70}, the function defined by
\[B_\pi(g) = \frac{1}{|N(E)|} \sum_{n \in N(E)} \psi^{-1}(n) \chi_\pi(gn)\]
belongs to $\mathcal W(\pi,\psi)$. Note that $B_\pi(I_n)=1$. This is called the Bessel function of $\pi$.

Let $\psi$ be a non-degenerate character of $N(E)$ such that $\psi = \psi^\iota$. Note that such a character can be constructed starting from a non-trivial additive character $\psi_0$ of $F$. Define the character $\psi$ of $E$ to be 
\[\psi(x) = \begin{cases} \psi_0({\rm Trace}(x)) &\text{if $\iota = \sigma$,} \\
\psi_0({\rm Trace}(\Delta x))   &\text{if $\iota=\tau$,}
\end{cases}\]  
where $\Delta \in E$ is of trace zero. Continue to denote by $\psi$ the character of $N(E)$ given by
\[\psi((x_{ij})) = \psi(\sum_{i=1}^{n-1} x_{i,i+1}).\]
Observe that $\psi = \psi^\iota$. Now if $\pi$ is irreducible and generic such that $\pi \cong \pi^\iota$, we have $\mathcal W(\pi,\psi) = \mathcal W(\pi^\iota,\psi)$, by the uniqueness of the $\psi$-Whittaker model. We may realize the $\iota$-twisted intertwining operator $T_\iota$ on $\mathcal W(\pi,\psi)$ as 
\[T_\iota (W) = W^\iota,\]
where $W^\iota(g) = W(g^\iota)$. 

The most crucial property of the Bessel function that we will appeal to is the following lemma by Gelfand \cite[Proposition 4.9]{gel70} (see also \cite[Lemma 2.14]{nie14}). Let $n = n_1+ \dots +n_k$ be a partition of $n$ into positive terms. Let $a_1,\dots,a_k \in E^\times$. Let
\[g_{n_1,\dots,n_k}(a_1,\dots,a_k) = \left[\begin{array}{ccccc}
0 & 0 & \dots & 0 & a_1I_{n_1} \\
0 & 0 & \dots & a_2I_{n_2} & 0 \\
\dots & \dots & \dots & \dots & \dots \\
\dots & a_{k-1}I_{n_{k-1}} & 0 & 0 & 0 \\
a_kI_{n_k} & 0 & 0 & 0 & 0
\end{array} \right]\]

\begin{proposition}\label{prop-gelfand}
Let $\pi$ be an irreducible generic representation of $\GL_n(E)$ and let $B_\pi$ be its (normalized) Bessel function. If $B_\pi(aw) \neq 0$, where $a$ is a diagonal matrix and $w$ is a permutation matrix, then 
\[aw = g_{n_1,\dots,n_k}(a_1,\dots,a_k), \]
for some $a_i \in E^\times$, $n_i \in \mathbb N$, $1 \leq i \leq k$.  
\end{proposition}

\section{Proof of Theorem \ref{thm-main}}\label{pf-thm-main}

Let $\pi$ be an irreducible generic representation of $\GL_n(E)$ which is distinguished with respect to $G_\iota$. Let $\kappa$ be the involution opposite to $\iota$. By which we mean $\kappa = \tau$ if $\iota = \sigma$ and $\kappa = \sigma$ if $\iota = \tau$. Thus, $\pi \cong \pi^\kappa$. Recall that $G_\iota =\GL_n(F)$ when $\iota =\sigma$ and $G_\iota = {\rm U}(n,E/F))$ when $\iota=\tau$. The proof is uniform and does not distinguish these two cases. Let $\psi: N(E) \rightarrow \mathbb C^\times$ be a non-degenerate character such that $\psi = \psi^\kappa$. Let $B_\pi \in \mathcal W(\pi,\psi)$ be the Bessel function of $\pi$ in its $\psi$-Whittaker model.

Our interest is in evaluating
\[\lambda_\iota(B_\pi) = \frac{1}{|G_\iota |} \sum_{h \in G_i} B_\pi(h).\]
Note that $\lambda_\iota$ is well-defined since $\psi = \psi^\kappa$. Let $X_\kappa$ be the set of norm one elements with respect to the involution $\kappa$; i.e., 
\[X_\kappa = \{g \in \GL_n(E) \mid gg^\kappa =1\}.\]
The strategy is to consider 
\[\mu_\kappa(B_\pi) = \frac{1}{|G_\iota |} \sum_{g \in X_\kappa} B_\pi(g),\]
which is easier to evaluate and to claim that $\lambda_\iota(B_\pi) = \mu_\kappa(B_\pi)$.

Let $A$ be the maximal torus of $\GL_n(E)$ consisting of diagonal matrices and let $W$ be the group of permutation matrices representing the Weyl group of $\GL_n(E)$. We first observe the following easy lemma.

\begin{lemma}\label{lemma1}
The following two sets are equal
\[\{aw \in G_\iota \cap AW \mid B_\pi(aw) \neq 0 \}=\{aw \in X_\kappa \cap AW \mid B_\pi(aw) \neq 0 \}. \]
\end{lemma}

\begin{proof}
By Proposition \ref{prop-gelfand}, $aw$ is of the form $g_{n,a}=g_{n_1,\dots,n_k}(a_1,\dots,a_k)$, however 
${^t}g_{n,a}=Jg_{n,a}J^{-1}$ and the result follows.
\end{proof}

The next lemma is about the shape of the Bruhat cells.

\begin{lemma}\label{lemma2}
Let $\iota$ be $\sigma$ or $\tau$. If $g \in X_i$ then $g=nawn^{-\iota}$ for some $n \in N(E), a \in A, w \in W$. Moreover $ww^\iota=1$.
\end{lemma}

\begin{proof}
One could argue as in the proof of \cite[Proposition 3]{fli92} invoking Lang's theorem but for the sake of brevity it is easier to use the said result instead. According to \cite[Proposition 3]{fli92}, 
\[\GL_n(E) = \bigsqcup N(E)A\eta_wG_\iota,\]
where the union is over $w \in W$ such that $ww^\iota=1$ and $\eta_w$ satisfies $\eta_w\eta_w^{-\iota} = w$, and the double coset is independent of the choice of the representative. Now, if $g \in X_\iota$ then $g =\alpha \alpha^{-\iota}$ for $\alpha \in \GL_n(E)$, by Lang's theorem, and writing $\alpha = nb\eta_w h$ , it follows that 
\[g = (nb\eta_w h)(nb\eta_w h)^{-\iota} = nbwb^{-\iota}n^{-\iota}=nawn^{-\iota},\]
for some $a \in A$.
\end{proof}

Now we can see the equality of the linear forms $\lambda_\iota$ and $\mu_\kappa$ on the Bessel function.

\begin{lemma}\label{lemma3}
Let $\iota,\kappa \in \{\sigma,\tau\}$ be opposite involutions. Let $\psi=\psi^\kappa$. Let $\lambda_\iota$ and $\mu_\kappa$ be the linear forms defined on $\mathcal W(\pi,\psi)$ by
\[\lambda_\iota(W) = \frac{1}{|G_\iota |} \sum_{h \in G_i} W(h) ~\&~ \mu_\kappa(W) = \frac{1}{|G_\iota |} \sum_{g \in X_\kappa} W(g). \] 
Then,
\[\lambda_\iota(B_\pi) = \mu_\kappa(B_\pi).\]
\end{lemma}

\begin{proof}
By the Bruhat decomposition for the group $G_\iota$ and Lemma \ref{lemma2} and its proof, we are led to look at $h = n_1awn_2^{-1}$ and $g=nawn^{-\kappa}$ with $B_\pi(aw) \neq 0$ of the following form: $w \in W$ is such that $ww^{\kappa}=1$ 
and $a \in A$ is of the form $bwb^{-\kappa}w^{-1}$ for $b$ diagonal, $n_1$ and $n_2$ are from the subgroup $N_\iota$ of unipotent upper triangular matrices in $G_\iota$ and $n \in N(E)$. Since $\psi=\psi^\kappa$, 
\[\psi(n_1)=\psi(n_2^{-1})=1 ~\&~ \psi(nn^{-\kappa})=1,\]
the proof of the lemma boils down to check that 
\[\{n_1aw n_2^{-1} \mid B_\pi(aw) \neq 0, aw \in G_\iota \cap AW, n_1,n_2 \in N_\iota\}\]
and
\[\{nawn^{-\kappa} \mid B_\pi(aw) \neq 0, aw \in X_\kappa \cap AW, n \in N(E) \}\]
have the same cardinality. Because $aw \in G_\iota \cap AW$ if and only if $aw \in X_\kappa \cap AW$ by Lemma \ref{1}, and because 
$N_\iota\times N_\iota$ and $N(E)$ have the same cardinality (when $\iota=\tau$ this can be checked at the Lie algebra level, using the exponential), it is enough to show that the stabilizers of $aw$ in $N_\iota\times N_\iota$ and $N(E)$ with $B_\pi(aw) \neq 0$, for $w$ such that $ww^{\kappa}=1$ and $a=bwb^{-\kappa}w^{-1}$ have the same cardinality. Noticing that 
$nawn^{-\kappa}=aw$ if and only if $b^{-1}nb$ fixes $w$, and that $n_1awn_2^{-1}=aw$ if and only if $(a^{-1}n_1a,n_2)$ fixes $w$, we can suppose that $a=1$. To conclude, denoting by $\ell$ the length on $W$ with respect to the set of simple reflections, one checks that 
\[|\mathrm{Stab}_{N_\iota\times N_\iota}(w)|=|\mathrm{Stab}_{N(E)}(w)|=q^{{n \choose 2}-\ell(w)}.\]

\end{proof}

To complete the proof of Theorem \ref{thm-main} it remains to evaluate $\mu_\kappa(B_\pi)$. We state this as the next lemma.

\begin{lemma}\label{lemma4}
Let $\pi$ be an irreducible generic representation of $G= \GL_n(E)$ which is a base change lift of a representation $\rho$ of $G_\kappa$. Then,
\[\sum_{g \in X_\kappa} B_\pi(g) = \frac{\dim \rho}{\dim \pi} \cdot  |X_\kappa| \]
\end{lemma}

\begin{proof}
Since $\rho$ base changes to $\pi$, we have the character identity
\[{\rm Trace~}[\pi(g)T_\kappa] = \chi_\rho(gg^\kappa),\]
for any $g \in \GL_n(E)$. Now consider the operator 
\[T = \sum_{g \in X_\kappa} \pi(g)T_\kappa \]
defined on the space of $\pi$. If $x \in \GL_n(E)$, observe that
\begin{align*}
T\pi(x) &= \sum_{g \in X_\kappa} \pi(g)T_\kappa \pi(x) \\
&= \sum_{g \in X_\kappa} \pi(g)\pi(x^\kappa)T_\kappa = \sum_{g \in X_\kappa} \pi(gx^\kappa)T_\kappa \\
&= \sum_{y \in G_\kappa \backslash \GL_n(E)} \pi(y^{-1}y^\kappa x^\kappa)T_\kappa \\
&=  \sum_{y \in G_\kappa \backslash \GL_n(E)} \pi(x(yx)^{-1}(yx)^\kappa)T_\kappa \\ 
&= \sum_{g \in X_\kappa} \pi(xg) T_\kappa = \sum_{g \in X_\kappa} \pi(x)\pi(g) T_\kappa \\
&= \pi(x)T.
\end{align*}
It follows, by Schur's lemma, that $T$ is a scalar, say $c(\pi)$. Thus, we have,
\[ \sum_{g \in X_\kappa} \pi(g)T_\kappa = c(\pi) \cdot I.\]
Taking trace on both sides,
\[\sum_{g \in X_\kappa} \chi_\rho(gg^\kappa) = c(\pi) \cdot \dim \pi.\]
But the left hand side is 
\[\sum_{g \in X_\kappa} \chi_\rho(1) = \dim \rho \cdot |X_\kappa|.\]

Now we have evaluated the scalar $c(\pi)$ to be the right hand side of the identity in the lemma. To complete the proof we apply the operator identity 
\[ \sum_{g \in X_\kappa} \pi(g)T_\kappa = c(\pi) \cdot I\]
to the Bessel function $B_\pi$ and evaluate the resulting function at the identity matrix. The right hand side gives
\[c(\pi)B_\pi(I_n) = c(\pi) \cdot 1 = c(\pi),\]
whereas the left hand side gives, by observing that $T_\kappa B_\pi = B_\pi$ (since $\psi = \psi^\kappa$ and by the uniqueness of the Bessel function),
\[\sum_{g \in X_\kappa} \pi(g)T_\kappa B_\pi (I_n) = \sum_{g \in X_\kappa} \pi(g)B_\pi(I_n) = \sum_{g \in X_\kappa}B_\pi(g). \]
This completes the proof of the lemma.
\end{proof}

\begin{proof}[Proof of Theorem \ref{thm-main}]
Theorem \ref{thm-main} follows by clubbing Lemma \ref{lemma3} and Lemma \ref{lemma4}, and by observing that $|X_\kappa| = \frac{|\GL_n(E)|}{|G_\kappa|}$. 
\end{proof}

\begin{remark}\label{split}
In the split case, i.e., when $E=F \oplus F$, we will be looking at $\pi \otimes \pi$ or $\pi \otimes \pi^\vee$ and the base changing representation is $\pi$. So the quotient of the dimensions on the right hand side of Theorem \ref{thm-main} is $1/\dim \pi$. The quotient of the group orders cancel out as 
\[\GL_n(E) = \GL_n(F) \times \GL_n(F)\]
and the subgroups involved are 
\[H = \{(g,g) \mid g \in \GL_n(F)\}\]
and
\[U = \{(g,w{^t}g^{-1}w^{-1}) \mid g \in GL_n(F)\}.\]
Note that 
\[B_{\pi,\psi}(w{^t}g^{-1}w^{-1})=\overline{B_{\pi,\psi}(g)},\]
and thus in both the cases, i.e., distinction with respect to $H$ and with respect to $U$, Theorem \ref{thm-main} takes the form
\[\frac{1}{|\GL_n(F)|}\sum_{g \in \GL_n(F)}B_{\pi,\psi}(g)\overline{B_{\pi,\psi}(g)} = \frac{1}{\dim \pi}.\]
However, this can be proved by rather elementary means as well and it does not require the strategy employed in the paper. Indeed this identity follows immediately from the definitions by an application of \cite[Lemma 5.1]{gel70}. 
\end{remark}

\begin{remark}\label{rmk}
An appealing way of stating the identity of Theorem \ref{thm-main} is as follows:
\[\frac{\dim \pi}{\dim \rho_\kappa} \cdot \lambda_\iota(B_\pi) = \frac{\dim \pi}{\dim \rho_\iota} \cdot \lambda_\kappa(B_\pi) = \frac{|G|}{|G_\iota||G_\kappa|}.\]
Both sides have three quantities related respectively to $G$, $G_\iota$, and $G_\kappa$.
\end{remark}

\begin{remark}\label{mu}
The crux of the proof of Theorem \ref{thm-main} is the equality of the linear forms $\lambda_\iota$ and $\mu_\kappa$ defined on $\mathcal W(\pi,\psi)$. The form $\lambda_\iota$ is obviously $G_\iota$-invariant. Though not that obvious it is not difficult to see that the form $\mu_\kappa$ is in fact a $\psi$-Whittaker functional. Indeed,
\begin{align*}
\mu_\kappa(\pi(n)W) &= \frac{1}{|G_\iota|} \sum_{g \in X_\kappa} W(gn) \\
&= \frac{1}{|G_\iota|} \sum_{g \in X_\kappa} W(n^{\kappa} (n^{-\kappa} gn)) \\
&= \frac{1}{|G_\iota|} \sum_{g \in X_\kappa} W(n^{\kappa} g),
\end{align*}
by sending $g$ to $n^\kappa g n^{-1} \in X_\kappa$. Thus,
\[\mu_\kappa(\pi(n)W) = \psi^\kappa(n) \mu_\kappa(W) = \psi(n) \mu_\kappa(W), \]
since $\psi = \psi^\kappa$. Thus we are equating the value at the $\psi$-Bessel function of a $G_\iota$-invariant functional and a $\psi$-Whittaker functional.
\end{remark}

\begin{remark}\label{reg}
Though it is not obvious to conclude $\lambda_\iota(B_\pi) \neq 0$ for each irreducible generic distinguished representation it is quite easy to conclude that not all $\lambda_\iota(B_\pi)$ are $0$. This follows from orthogonality relations as follows.
From the identity,
\[B_\pi(g) = \frac{1}{|N(E)|} \sum_{n \in N(E)} \psi^{-1}(n) \chi_\pi(gn),\] 
it follows easily that 
\begin{align*}
\sum_{\pi \in \hat{G}} \dim \pi ~\lambda_\iota(B_\pi) &=  \frac{1}{|N|} \frac{1}{|G_\iota|} \sum_{\pi \in \hat{G}} \sum_{n \in N} \sum_{h \in G_\iota} \dim \pi ~\psi^{-1}(n)\chi_\pi(hn) \\
&=  \frac{1}{|N|} \frac{1}{|G_\iota|} \sum_{n \in N} \sum_{h \in G_\iota} \psi^{-1}(n) \left(\sum_{\pi \in \hat{G}} \dim \pi ~ \chi_\pi(hn) \right) \\
&= \frac{1}{|N|} \frac{1}{|G_\iota|} \sum_{n \in N} \sum_{h \in G_\iota} \psi^{-1}(n) \chi_{\rm reg}(hn) \\
&= \frac{|G|}{|G_\iota||N_\iota|},
\end{align*}
where ${\rm reg}$ is the regular representation of $G$ and $\hat{G}$ is the set of all irreducible representations of $G$. Note that by Theorem \ref{thm-main},
\begin{align*}
\sum_{\pi \in \hat{G}} \dim \pi ~\lambda_\iota(B_\pi) &= \sum_{\pi \in \hat{G}} \dim \pi ~ \frac{\dim \rho}{\dim \pi} \frac{|G|}{|G_\iota||G_\kappa|} \\
&= \frac{|G|}{|G_\iota||G_\kappa|} \sum_{\pi \in \hat{G}^{\rm gen}} \dim \rho \\
&= \frac{|G|}{|G_\iota||G_\kappa|} \sum_{\rho \in \hat{G}^{\rm gen}_\kappa} \dim \rho \\
&= \frac{|G|}{|G_\iota||G_\kappa|} \frac{|G_\kappa|}{|N_\kappa|} \\
&= \frac{|G|}{|G_\iota||N_\iota|},
\end{align*}
as $|N_\iota| = |N_\kappa|$. In the above, $\hat{H}^{\rm gen}$ stands for the set of irreducible generic representations of $H$. Note that the injectivity of the base change map is used in the third step. 
\end{remark}

\begin{remark}\label{spectral}
The Bessel function is known to be a spectral projector. That is, if $W$ is any $\psi$-Whittaker function on $G$, its projection to $\mathcal W(\pi,\psi)$ is given by
\[P_\pi W(x)  = \frac{\dim \pi}{|G|} \sum_{g \in G} W(g) B_\pi(xg^{-1}). \]
Now the most natural $\psi$-Whittaker function on the group $G$ is given by
\[W_1(g) = \begin{cases}
\psi(n) &\text{if $g = nh \in NG_\iota$,} \\
0 &\text{otherwise.}
\end{cases}
\]
Observe that
\[P_\pi W_1 = \dim \pi \cdot \frac{|N_\iota|}{|G|} \sum_{h \in G_\iota} \pi(h)B_\pi = \dim \pi \cdot \frac{|N_\iota||G_\iota|}{|G|} W_\pi.\]
Thus, Corollary \ref{cor-1} asserts that $W_1$ projects non-trivially to each irreducible generic distinguished representation $\pi$.
\end{remark}

\begin{remark}\label{rem-sl2}
As mentioned in Remark \ref{nov19}, the identity of Theorem \ref{thm-main} generalizes to irreducible generic representations of $G(E)$ which are in the image of the Shintani base change lift from $G^{op}(F)$ \cite{yang19}, showing in particular that they are distinguished. On the other hand, there are groups $G$ for which the converse implication is known not to hold, i.e., such that 
there is a $G(F)$-distinguished generic representation of $G$ which is not in the image of the base change map from $G^{op}(F)$. Here we take one such example for $G=\SL (2)$ and calculate the value of the natural invariant linear form on the Bessel function but it is not indicative of what to expect in a general situation. Take $\alpha$ the non-trivial quadratic character of $F^\times$ and $\tilde{\alpha}$ an extension of it to $E^\times$, and consider $\pi=\1\times \alpha\circ N_{E/F}=\1\times \alpha^{-1}\circ N_{E/F}$. It is a $\GL_2(F)$-distinguished irreducible principal series representation of $\GL_2(E)$ and $\pi \otimes \tilde{\alpha}=\tilde{\alpha} \times \tilde{\alpha}^{-\sigma} $. Let $\pi^\prime$ be the $\psi$-generic element of the restriction of $\pi$ to $\SL_2(E)$. Take $W\in \mathcal{W}(\pi^\prime,\psi)\subset \mathcal{W}(\pi,\psi)$. By Fourier inversion,
\begin{align*}
L(W)=\sum_{h \in \SL_2(F)} W(h) &= \frac{1}{|\GL_2(F)|} \sum_\chi \sum_{g\in \GL_2(F)} W(g) \chi(\det g)  \\
&=  \frac{1}{|\GL_2(F)|} \left[\sum_{g \in \GL_2(F)} W(g) + \sum_{g \in \GL_2(F)} W(g)\alpha(\det g) \right].
\end{align*}
By the formulas for $(\GL_2(E),\GL_2(F))$, the right hand side (for $W= B_{\pi,\psi} = B_{\pi^\prime,\psi}$) is 
\[\frac{1}{q+1} + \frac{1}{q-1} = \frac{2q}{q^2-1}.\] 
Indeed this follows since the principal series representation $\chi^{-1} \times \chi^\sigma$ is the base change of the principal series representation ${\rm Ind}_{E^\times}^{{\rm U}(2)} \chi$, which has dimension $q+1$, and the principal series representation $\chi_1 \times \chi_2$, with $\chi_i|_{F^\times}=1$, is the base change of a cuspidal representation of ${\rm U}(2)$ of dimension $q-1$.
\end{remark}

\begin{proof}[Proof of Theorem \ref{thm-scalar}]
It is enough to evaluate the linear forms $\lambda$ and $\ell$ at any one Whittaker vector and we do this on the Bessel function. Now the proof follows by clubbing Theorem \ref{thm-main} with Proposition \ref{prop-gelfand}. Indeed, by Proposition \ref{prop-gelfand},
\[\ell(B_\pi)= \frac{1}{|\GL_n(F)|} \cdot |N(F)|,\]
and thus by Theorem \ref{thm-main}, $\lambda$ and $\ell$ differ by 
\[\frac{|\GL_n(E)|}{|N(F)||{\rm U}(n,E/F)|} = \frac{|\GL_n(E)/N(E)|}{|{\rm U}(n,E/F)/N(F)|}.\]
\end{proof}

\section{Comparison with the $p$-adic case}\label{bz}

In this section $K/k$ is a quadratic extension of $p$-adic fields and we consider the symmetric space $(\GL_n(K),\GL_n(k))$. Let $N$, $P$, and $\psi$ have similar meanings as in \S \ref{intro}. For details we refer to \cite{am17} and the references therein. 

An irreducible admissible representation $\pi$ of $\GL_n(K)$ is $\GL_n(k)$ distinguished if Hom$_{\GL_n(k)}(\pi,1) \neq 0$. As in the finite field case it is known that the above space is one dimensional when $\pi$ is distinguished. We also know that there is an integral representation for the invariant linear form and that it can be realized on the Whittaker space $\mathcal W(\pi,\psi)$ by
\[\ell(W)=\int_{N(k)\backslash P(k)}W(p)dp.\]
The linear form $\ell$ is known to be not identically zero.

Suppose $\pi$ an irreducible admissible representation $\GL_n(K)$ which is distinguished with respect to $\GL_n(k)$. The key input in concluding that the obviously $P(k)$-invariant linear form $\ell$ on $\pi$ is in fact $\GL_n(k)$-invariant is the following identity 
\[{\rm Hom}_{\GL_n(k)}(\pi,1)={\rm Hom}_{P(k)}(\pi,1),\]
due to Youngbin Ok.

Ok's identity is the relative version of a well-known result due to Bernstein on $P$-invariant distributions \cite[Theorem A]{ber84}. As Bernstein remarks, there are several places in this work where the fact that the underlying field is $p$-adic and not finite plays a crucial role \cite[p. 53, p. 62]{ber84}. 

Similarly, Ok's identity is also not true over a finite field. A simple counterexample may be constructed by looking at a distinguished principal series representation of $\GL_2(E)$. So let $\pi = {\rm Ps}(\chi_1,\chi_2)$ be the principal series representation of $\GL_2(E)$ induced from the character of the Borel subgroup given by 
\[\chi \left(\left(\begin{array}{cc} a & b \\ 0 & d \end{array} \right) \right) = \chi_1(a)\chi_2(d),\]
where $\chi_1 \neq \chi_2$ are characters of $E^\times$. We assume $\chi_1|_{F^\times} = \chi_2|_{F^\times} =1$ so that $\pi$ is distinguished with respect to $\GL_2(F)$. By Mackey theory,
\[{\rm Ps}(\chi_1,\chi_2)|_{_{GL_2(F)}} = {\rm Ind}_{E^\times}^{GL_2(F)}[\chi_1\chi_2^q] \bigoplus {\rm Ps}(\chi_1|_{_{F^\times}},\chi_2|_{_{F^\times}}),\]
and by our assumption the second summand is ${\rm Ps}(1,1)$ which is the direct sum of the trivial representation and the Steinberg representation of $\GL_2(F)$. The first summand is the direct sum of all the twists of the Steinberg representation. Therefore $\dim {\rm Hom}_{P(F)}(\pi,1) = 3$ whereas $\dim {\rm Hom}_{\GL_2(F)}(\pi,1) = 1$.

In fact, in the $p$-adic case, by \cite[Theorem 1.1]{akt04} and more generally by \cite[Propositions 2.4 and 2.5]{mat14}, one has 
$\dim {\rm Hom}_{P(F)}(\pi,1)=1$ for any irreducible unitary representation of $\GL_n(E)$, even if it is not distinguished. Again the above example shows that it is not true when $F$ is finite, and the reason is that more factors of the Bernstein-Zelevinsky filtration 
can support a $P(F)$-invariant linear form in the setting of the paper because the absolute value is trivial. 

We now take a closer look at this phenomenon of higher multiplicity for $P(F)$-invariant vectors. As in the $p$-adic case, the space of such vectors is well understood thanks to the theory of Bernstein-Zelevinsky derivatives. We thus introduce the so called Bernstein-Zelevinsky (\cite{bz76}) functors $\Phi^-$, $\Phi^+$, $\Psi^-$, and $\Psi^+$, for representations of the linear group and its mirabolic subgroup over finite fields, and recall their basic properties. Of course, the arguments of Bernstein and Zelevinsky can considerably be simplified due to the fact that we work with finite groups, for example there is no need to introduce sheafs. We will thus give statements without proofs for most of their properties, the proofs being left as an easy but informative exercise for the interested reader. 

We denote by $R(G)$ the category of finite dimensional complex representations of a finite group $G$. We recall that if $H$ is a subgroup of $G$, if $(\pi,V)$ is an object in $R(G)$, and if $\theta$ is a character of $G$, then one has 
\[V = V^{H,\theta} \oplus V(H,\theta),\]
where $V^{H,\theta}$ is the vectors in $v$ which transform by $\theta$ under the action of $H$, and \[V(H,\theta)=\langle \pi(h)v - \theta(h)v, h\in H, v\in V \rangle.\]
Denoting by $N_G(\theta)$ the subgroup of the normalizer $N_G(H)$ of $H$ which stabilizes $\theta$, 
$V_{H,\theta}=V/V(H,\theta)$ and $V^{H,\theta}$ are canonically isomorphic as $N_G(\theta)$-modules, and we will use both the models, depending on which is more convenient for the computations. In particular, a representation admits a non zero $(H,\theta)$-equivariant linear form on its space if and only if $V^{H,\theta}\neq 0$.

Let $P_n = P_n(E)$ be the mirabolic subgroup of $G_n = \GL(n,E)$ so that $P_n=G_{n-1}U_n$, where 
\[U_n = \left\{ \begin{pmatrix} I_{n-1} & X \\ & 1 \end{pmatrix} \subset G_n \right\}.\] Here, $G_0$ is the trivial group by convention. We denote the character group of $U_n$ by $\widehat{U}_n$. Any $\pi \in R(U_n)$ can be written as 
\[\pi = \displaystyle{\bigoplus_{\theta\in \widehat{U}_n}} \pi^{U_n,\theta}.\]
Note that if $\pi$ is in fact the restriction to $U_n$ of an object in $R(P_n)$ then all $\pi^{U_n,\theta}$, with $\theta \neq \mathbf{1}$, are isomorphic $U_n$-modules, as all the non-trivial characters of $U_n$ are conjugate under $G_{n-1}$. We denote by $\theta_n$ the character of $U_n$ defined by 
$\theta_n (u)=\psi(u_{n-1,n})$. Notice that $N_{P_n}(\theta_n)=P_{n-1}$. We set:

$\Psi^+:R(G_{n-1})\rightarrow R(P_n), \mbox{~with~} \Psi^+(V)=V, \mbox{~and letting~} U_n \mbox{~act trivially on~} V. $ 

$\Psi^-:R(P_{n}) \rightarrow R(G_{n-1}), \mbox{~with~} \Psi^-(V)=V^{U_n,\1} \simeq V_{U_n,\1}.$ 

$\Phi^+:R(P_{n-1})\rightarrow R(P_n), \mbox{~with~} \Phi^+(V) = {\rm Ind}_{P_{n-1}.U_n}^{P_n}(V\otimes \theta_n).$ 

$\Phi^-:R(P_{n})\rightarrow R(P_{n-1}), \mbox{~with~} \Phi^-(V) = V^{U_n,\theta_n} \simeq V_{U_n,\theta_n}.$

We have the following basic proposition (cf. \cite[Proposition 3.2]{bz77}). 

\begin{proposition}\label{prop-derivatives-basic}
We have the relations:
\begin{enumerate}
\item\label{1} $\Phi^-,\Phi^+,\Psi^+,\Psi^-$ are exact and commute with taking the contragredient.
\item\label{2} $\Phi^+$ and $\Phi^-$ are left and right adjoint of each other, and so are $\Psi^+$ and $\Psi^-$.
\item\label{3} $\Phi^- \circ \Psi^+=0$ and $\Psi^- \circ \Phi^+=0$.
\item\label{4} $\Phi^- \circ \Phi^+\simeq {\rm Id}$ and $\Psi^- \circ \Psi^+\simeq {\rm Id}$.
\item\label{5} For any $\pi \in R(P_n)$, we have 
\[\pi \simeq \Phi^+ \circ \Phi^-(\pi) \oplus \Psi^+ \circ \Psi^-(\pi),\] 
where 
\[\Psi^+ \circ \Psi^-(\pi)\simeq \pi^{U_n,1}\] and 
\[\Phi^+ \circ \Phi^-(\pi)\simeq \underset{\theta\in \widehat{U}_n-\{\1\}}{\bigoplus}\pi^{U_n,\theta}.\]
\end{enumerate}
\end{proposition}

Now if $\tau$ is in $R(P_n)$ (or in $R(G_n)$, considering its restriction to $P_n$), we define 
$\tau^{(k)}$ to be the representation $\Psi^- \circ (\Phi^-)^{k-1}$ of $G_{n-k}$. It follows from Proposition \ref{prop-derivatives-basic}, as in \cite[\S 3.5]{bz77}, that we have the following decomposition of representations of $P_n$ by $P_n$-submodules.

\begin{proposition}\label{prop-BZ-filtration}
For $\tau \in R(P_n)$, we have:
\[\tau \simeq \bigoplus_{k=1}^{n} (\Phi^+)^{k-1} \circ \Psi^+(\tau^{(k)}).\]
\end{proposition}

We now state and prove the finite version of a very useful result of Kable \cite[Proposition 1]{kab04} (see also \cite[Appendix]{fli93}). Note that the proof boils down to elementary Mackey theory whereas Kable has to use Bruhat's method on invariant distributions in the $p$-adic setting. In the following, a prime in the superscript denotes the $F$-points; $P_n^\prime = P_n(F)$, $U_n^\prime = U_n(F)$, etc.

\begin{proposition}\label{prop-kable}
If $\tau$ is a representation of $P_{n-1}$, then 
\[{\rm Hom}_{P_{n}^\prime}(\Phi^+(\tau),\1)\simeq {\rm Hom}_{P_{n-1}^\prime}(\tau,\1).\]
Similarly, \[{\rm Hom}_{P_{n}^\prime}(\Phi^+(\rho),\1)\simeq {\rm Hom}_{G_{n-1}^\prime}(\rho,1).\] 
In particular, $(\Phi^+)^{k-1}\Phi^+ (\rho)$ is $P(F)$-distinguished if and only if $\rho$ is distinguished. Moreover, 
\[{\rm Hom}_{P_n^\prime}({\rm Ind}_{N_n}^{P_n}\psi,\1) \simeq \mathbb{C}.\]
\end{proposition}

\begin{proof}
By Mackey theory, we have
\[\Phi^+(\tau)|_{P_n^\prime} = \bigoplus_{s \in P_{n-1}.U_n \backslash P_n/P_n^\prime} 
{\rm Ind}_{(P_{n-1}.U_n)^s \cap P_n^\prime}^{P_n^\prime}(\tau \otimes \theta_n)^s.\]
Hence by Frobenius reciprocity, we get:
\[{\rm Hom}_{P_{n}^\prime}(\Phi^+(\tau),\1)\simeq \prod_s {\rm Hom}_{(P_{n-1}.U_n)^s \cap P_n^\prime}((\tau\otimes \theta_n)^s,\1).\]
Notice that $(P_{n-1}.U_n)^s \cap P_n^\prime = P_{n-1}^s.U_n \cap P_n^\prime$ and $(\tau\otimes \theta_n)^s = \tau^s\otimes \theta_n^s$, because $s$ normalizes $U_n$. We also have 
\[P_{n-1}.U_n \backslash P_n/P_n^\prime \simeq P_{n-1}\backslash G_{n-1}/G_{n-1}^\prime,\] 
and the map $g \mapsto \theta_n^g$ identifies $P_{n-1}\backslash G_{n-1}$ with $\widehat{U}_n-\{\1\}$. Hence $\theta_n^s$ describes the set of $G_n^\prime$-orbits in $\widehat{U}_n-\{\1\}$ when $s$ varies in $P_{n-1}.U_n \backslash P_n/P_n^\prime$. Now it is easily checked that the orbit of $\theta_n$ is exactly the set of non-trivial characters of $U_n$ which are trivial on $U_n^\prime$. In particular,  $\theta_n^s$ is not trivial on $U_n^\prime$ when $s \neq 1$. This implies that for $s \neq 1$, 
\[{\rm Hom}_{(P_{n-1}.U_n)^s \cap P_n^\prime}(\tau^s \otimes \theta_n^s,\1) \subset {\rm Hom}_{U_n^\prime}(\tau^s \otimes \theta_n^s,\1) = 0.\]
Therefore,  
\[{\rm Hom}_{P_{n}^\prime}(\Phi^+(\tau),\1) \simeq {\rm Hom}_{(P_{n-1}.U_n)^\prime}(\tau\otimes \theta_n,\1) \simeq {\rm Hom}_{P_{n-1}^\prime.U_n^\prime}(\tau\otimes \theta_n,\1) \simeq 
{\rm Hom}_{P_{n-1}^\prime}(\tau,\1).\]
The last statement follows at once from the first, since 
\[{\rm Ind}_{N_n}^{P_n}\psi=(\Phi^+)^{n-1} \circ \Psi^+(\1).\]
The other statements are immediate. 
\end{proof}

We give some applications of this proposition, notice already that it immediately implies that ${\rm Hom}_{{P_n}^\prime}(\pi,1)$ has at least dimension one when $\pi$ is generic. 
First we classify generic distinguished representations with multiplicity one for 
$P_n^\prime$-invariant linear forms (we recall that it is all of them in the $p$-adic setting). We use $\times$ to denote parabolic induction.

\begin{corollary}\label{cor-P-multiplicity one}
Let $\pi$ be a distinguished generic representation of $G_n$. Then \[{\rm Hom}_{{P_n}^\prime}(\pi,1)\simeq \mathbb{C}\] if and only if 
$\pi$ is either cuspidal distinguished, or $\pi\simeq \rho^\vee\times \rho^\sigma$ for $\rho$ cuspidal with $\rho^\vee \not \simeq \rho^\sigma$.
\end{corollary}

\begin{proof}
A generic representation of $G_n$ can uniquely be written under the form of a commutative product 
\[\pi=\St(m_1,\rho_1)\times \dots \times \St(m,\rho_r),\] 
where $\rho_i$ is cuspidal and not isomorphic to $\rho_j$ for $i\neq j$, and 
$\St(m_i,\rho_i)$ is the generic summand of $\rho_i^{m_i}=\rho_i\times\dots \times \rho_i$. It is moreover distinguished (i.e., $\pi^\vee=\pi^\sigma$) if and only if one can order 
the $\St(m_i,\rho_i)$'s so that $\St(m_{i+1},\rho_{i+1})^\sigma\simeq \St(m_{i},\rho_{i})^\vee$ for $i=1,\dots,s$ for $s\leq \lfloor r/2 \rfloor$ and 
 $\St(m_k,\rho_k)$ is distinguished (i.e., $\rho_k$ is $\sigma$-self-dual) for $k\geq 2s+1$.
Now take $\rho$ a cuspidal representation of $G_l$, and $n=ml$. The nonzero derivatives of $\St(m,\rho)$ are the $\St(m,\rho)^{(al)}$ for $a=1,\dots,m$; indeed they are by definition submodules of the Jacquet modules $\pi^{U_{i,n-i}}$ corresponding to standard unipotent subgroups $U_{i,n-i}$ 
of type $(n-i,i)$ for $i=1,\dots, n$, and the only such nonzero Jacquet modules are for $i=al$. Moreover, as 
$\St(m,\rho)$ is the generic summand of $\St(m-a,\rho)\times \St(a,\rho)$, we deduce by adjointness that $\St(m-a,\rho)\otimes \St(a,\rho)$ is a summand of 
$\St(m,\rho)^{U_{n-al,al}}$, hence that $\St(m-a,\rho)$ is a summand of $\St(m,\rho)^{(al)}$. Notice that if $\rho$ is distinguished, then 
$\St(k,\rho)$ as well for all $k$. 
The above discussion, together with the given shape of distinguished generic representations, the Leibniz rule for derivatives (\cite[Lemma 4.5]{bz77}, the proof of which applies in our setting), 
Proposition \ref{prop-BZ-filtration}, and Proposition \ref{prop-kable}, proves the result.
\end{proof}

We now give a general example of high multiplicity of $P(F)$-invariant linear forms in irreducible representations. 

\begin{example}\label{example}
Let $\rho_i$ be a representation of $G_{n_i}(E)$, cuspidal and distinguished by $G_{n_i}(F)$. Suppose moreover that for $i\neq j$, 
one has $\rho_i\not \simeq \rho_j$. Then the representation \[\pi=\rho_1\times \dots \times \rho_r\] of $G_n(E)$ (with $n=\sum_i n_i$) is irreducible, generic and distinguished. We already noticed that the nonzero derivative of a cuspidal representation $\rho$ of $G_n$ is $\rho^{(n)}$, which is isomorphic to $\mathbb{C}$ because cuspidal representations are generic (see \cite[Theorem 8]{gg62} which is easily recovered using derivatives). Hence by the Leibniz rule for derivatives, 
the non-zero derivatives $\pi^{(k)}$ of $\pi$ are obtained for $k$ of the form $n_{i_1}+\dots+n_{i_l}$ for $i_1<\dots<i_l$, in which case $\pi^{(k)}=\oplus_{j=1}^{m_k} \pi_{k,j}$, where $m_k$ is the number of sequences $i_1<\dots<i_l$ satisfying 
$n_{i_1}+\dots+n_{i_l}=k$, and $\pi_{k,j}$ is the irreducible representation $\prod_{i\neq i_1,\dots,i_l} \rho_i$ corresponding 
to the $j$-th sequence satisfying the property. All $\pi_{k,j}$ are again distinguished, and so 
${\rm Hom}_{P_{n}(F)}((\Phi^+)^{k-1}\Phi^+ \pi_{k,j},1)$ is isomorphic to $\mathbb{C}$ by Proposition \ref{prop-kable}. Finally  
Proposition \ref{prop-BZ-filtration} tells us that 
\[\dim {\rm Hom}_{P_{n}(F)}(\pi,1) = \sum_{k=1}^n m_k.\]
\end{example}

\section{The case $G=\SL(n)$}\label{sln}

The symmetric space $(\SL_n(K),\SL_n(k))$, where $K/k$ is a quadratic extension of $p$-adic fields, is well understood \cite{ap03,ap18}. In particular, if $\pi$ is an irreducible admissible generic representation of $\SL_n(K)$ and if some member of the $L$-packet of $\pi$ is distinguished with respect to $\SL_n(k)$, then it follows that the $\SL_n(k)$-distinguished representations of the $L$-packet of $\pi$ are precisely the representations which are $\psi$-generic \cite[Theorem 5.6 (2)]{ap18}. A crucial ingredient to the proof of this statement is the explicit determination of the $\GL_n(k)$-invariant functional on the irreducible admissible representation of $\GL_n(K)$ which gives rise to the $L$-packet of $\pi$ (which can be assumed to be $\GL_n(k)$-distinguished) which we recalled in \S \ref{bz}. This was what was lacking in the case of finite fields (cf. \cite[Remark 4]{ap18}). All the other ingredients required are known to be true for finite fields as well \cite[\S 3]{ap18}. Thus, closely following the methods of \cite{ap18}, as a corollary to the main theorem of this paper, we get the following theorem.

\begin{theorem}\label{thm-sln}
Let $E/F$ be a quadratic extension of finite fields and let $\pi$ be an irreducible generic representation of $\SL_n(E)$ which is distinguished with respect to $\SL_n(F)$. Then an irreducible representation $\pi^\prime$ of $\SL_n(E)$ from the $L$-packet of $\pi$ is distinguished with respect to $\SL_n(F)$ if and only if $\pi^\prime$ is $\psi$-generic for a non-degenerate character $\psi$ of $N(E)/N(F)$. 
\end{theorem}

\begin{remark}
There should be a theorem for the symmetric pair $(\SL_n(E),{\rm SU}(n))$ analogous to Theorem \ref{thm-sln}. However the details which are checked for $(\SL_n(E),\SL_n(F))$ in \cite{ap18} will have to be checked. We leave these out as it will be a digression from the main theme of this paper. 
\end{remark}

\section{A $p$-adic analogue of Theorem \ref{thm-main}}\label{hii}

Let $K/k$ be a quadratic extension of $p$-adic fields. Our interest is in distinction for the symmetric pair $(\GL_n(K),\GL_n(k))$ for square-integrable 
representations which is related to stable (resp. unstable) base change from the quasi-split unitary group ${\rm U}(n,K/k)$ when $n$ is odd (resp. even), by the Flicker-Rallis conjecture, 
now a theorem by the work of Mok \cite{mok15}. In this section we prove a $p$-adic analogue of Theorem \ref{thm-main}, which is in fact a consequence of the formal degree conjecture of Hiraga, Ichino, and Ikeda \cite{hii08}.   

For unexplained notations in this section, we refer to \cite{akt04,kab04,hii08,am17}. We fix a non-degenerate additive character $\psi$ of $K/k$. Note that such a character arises from a non-trivial additive character $\psi_0$ of $k$ as 
\[\psi(x) = \psi_0({\rm Trace}(\Delta x)),\]  
where $\Delta \in K$ is of trace zero. We normalize the Haar measures as in \cite{hii08}.

The formal degree conjecture relates in a precise way the formal degree of an elliptic tempered representation to the absolute value of its adjoint $\gamma$-factor at $s=0$ \cite[Conjecture 1.4]{hii08}. We are interested in two cases of this conjecture, namely, for the groups $\GL_n(K)$ and ${\rm U}(n,K/k)$, when the representation is square-integrable. The $\GL_n(K)$ case is known and the case of ${\rm U}(n,K/k)$, which was known earlier for $n \leq 3$ \cite[Theorem 3.1 and Theorem 8.6]{hii08}, has been proved recently by Beuzart-Plessis for any $n$ in \cite[Corollary 5.4.4]{bp18b}. Note that \cite{hii08} employs the Langlands-Shahidi $\gamma$-factor whereas 
we argue with the Rankin-Selberg $\gamma$-factors but then these two different definitions do coincide, up to a root of unity, 
which is what is of relevance in this section, 
in the cases at hand. For the $\gamma$-factor for pairs the said equality is due to Shahidi (cf. \cite[Theorem 5.1]{sha84} whereas for the Asai $\gamma$-factor the equality up to a root of unity was proposed in \cite[Remark 3.5]{ar05}, however executing this strategy requires a careful normalization of the Rankin-Selberg $\gamma$-factor in the Asai case (cf. \cite[Theorem 9.29 and Remark 9.33]{akmss18}).

Let $\pi$ be a square-integrable representation of $\GL_n(K)$ which is distinguished with respect to $\GL_n(k)$. Let $\rho$ be the (unique) representation of ${\rm U}(n,K/k)$ that base changes to $\pi$. When $n$ is odd, we consider the stable base change, whereas we consider the unstable base change when $n$ is even. Notice that \cite[Section 8]{hii08} is written for $n$ odd so the unstable base change does not appear there. Notice as well that the $p$-adic unitary group we consider is defined with respect to a matrix $J^\prime$ with alternating $1$'s and $-1$'s on the anti-diagonal. We denote the formal degrees by $d(\pi)$ and $d(\rho)$ respectively. Note that these depend on the choices of the Haar measures which in turn depend on the choice of the additive characters. 

There are three $\gamma$-factors that will play a role; 
the one for the pair $(\pi,\pi^\vee)$, and the Asai and the twisted Asai $\gamma$-factors. We denote these respectively by 
$\gamma(s,\pi \times \pi^\vee,\psi)$, $\gamma(s,\pi,r,\psi)$, and $\gamma(s,\pi,r^\prime,\psi)$. Here, $r$ is the Asai representation and $r^\prime = r \otimes \omega_{K/k}$, where $\omega_{K/k}$ is the quadratic character of $k^\times$ associated to $K/k$.

Let the notation $\sim$ mean an equality up to a positive explicit constant that does not depend on the representations involved. We have
\[\left|\gamma(s,\pi,r,\psi)\right| \sim \left|\gamma^{\rm LS}(s,\pi,r,\psi_0)\right|, \] 
and
\[\left|\gamma(s,\pi,r^\prime,\psi)\right| \sim \left|\gamma^{\rm LS}(s,\pi,r^\prime,\psi_0)\right|,\]
where $\gamma^{\rm LS}(\cdot)$ denotes the Langlands-Shahidi definition of the corresponding gamma factors (cf. \cite[Theorem 9.29 and Remark 9.33]{akmss18}). 

Now the formulas for the formal degrees are as follows. By \cite[Theorem 3.1]{hii08},
\[d(\pi) =\frac{1-q^{-1}}{n} \left| \lim_{s \rightarrow 0} \frac{\gamma^{\rm LS}(s,\pi \times \pi^\vee, \psi)}{1-q^{-s}} \right|,\]
and, by \cite[Proposition 8.5]{hii08}, and by the forthcoming work of Beuzart-Plessis mentioned earlier, 
\[d(\rho) = \frac{1}{2} \left| \gamma^{\rm LS}(0,\pi,r^\prime,\psi_0) \right|.\] Thus, we have,
\begin{align}
d(\pi) &\sim \left| \lim_{s \rightarrow 0} s^{-1}\gamma(s,\pi \times \pi^\vee, \psi) \right|
\end{align}
and
\begin{align}
d(\rho) &\sim \left| \gamma(0,\pi,r^\prime,\psi) \right|.
\end{align}

Let $Z(s,W,\varphi)$ be the zeta integral that defines the Asai $\gamma$-factor via the Rankin-Selberg method \cite{fli93,kab04}. We have 
\cite[Proposition 2]{kab04},
\[Z(1-s,\widetilde{W},\widehat{\varphi}) = \gamma(s,\pi,r,\psi) Z(s,W,\phi),\]
where $\widetilde{W}(g)=W(J~{^t}g^{-1})$ and $\widehat{\varphi}$ is the Fourier transform of $\varphi$.
Since $\pi$ is square-integrable there are two natural candidates for $\GL_n(k)$-invariant linear forms on the Whittaker model $\mathcal W(\pi,\psi)$. One is the analogue of the linear form in Theorem \ref{thm-main} given by the convergent integral \cite{kab04}
\[\lambda(W) = \int_{k^\times N(k) \backslash \GL_n(k)} W(h)dh,\]
which is obviously $\GL_n(k)$-invariant but non-vanishing precisely when $\pi$ is distinguished, and the other, for which the finite field analogue is not very useful for the reasons detailed in \S \ref{bz}, is given by the convergent integral \cite{akt04,am17}
\[\ell(W) = \int_{N(k) \backslash P(k)} W(p)dp,\]
which is known to be always non-vanishing but $\GL_n(k)$-invariant precisely when $\pi$ is distinguished (cf. \S \ref{bz}).

It follows from the proof of \cite[Theorem 4]{kab04} that
\[\lim_{s \rightarrow 0} s Z(s,W,\varphi) = c_1 \cdot \varphi(0) \cdot \lambda(W),\]
where $c_1$ is a certain volume depending on the measures,
and from the proof of \cite[Theorem 1.4]{akt04} that
\[Z(1,\widetilde{W},\widehat{\varphi}) = c_2 \cdot \varphi(0) \cdot \ell(\widetilde{W}) = c_2 \cdot \varphi(0) \cdot \ell(W),\] 
where $c_2$ is a certain volume depending on the measures. Note that
the last equality in the above identity is highly non-trivial and this is \cite[Theorem 6.3]{am17}. 

Therefore we conclude that 
\begin{align*}
\lambda &= c \cdot \lim_{s \rightarrow 0} s \gamma(s,\pi,r,\psi)^{-1} \cdot \ell,
\end{align*}
where $c = c_2/c_1$. In other words,
\begin{align}\label{three}
\lambda &\sim \lim_{s \rightarrow 0} s \gamma(s,\pi,r,\psi)^{-1} \cdot \ell,
\end{align}

From the well-known factorization
\[\gamma^{\rm LS}(s,\pi \times \pi^\sigma,\psi_0 \circ {\rm Trace}_{K/k}) = \gamma^{\rm LS}(s,\pi,r,\psi_0) \gamma^{\rm LS}(s,\pi,r^\prime,\psi_0),\] it follows that
\begin{align}
\left|\gamma(s,\pi \times \pi^\sigma,\psi)\right | &\sim \left|\gamma(s,\pi,r,\psi)\right | \left|\gamma(s,\pi,r^\prime,\psi)\right |.
\end{align}
Noting that $\pi^\vee \cong \pi^\sigma$ (since $\pi$ is $\GL_n(k)$-distinguished), we get
\begin{align}\label{five}
\left| \lim_{s \rightarrow 0} s \gamma(s,\pi,r,\psi)^{-1} \right| &\sim \frac{\left|\gamma(0,\pi,r^\prime,\psi)\right |}{\left|\gamma(0,\pi \times \pi^\sigma,\psi)\right |} \sim \frac{d(\rho)}{d(\pi)}.
\end{align}
On the other hand,
\begin{align*}
\lim_{s \rightarrow 0} s \gamma(s,\pi,r,\psi)^{-1}  &= \frac{\displaystyle{\lim_{s \rightarrow 0}} ~s L(s,\pi,r)}{L(1,\pi^\vee,r)} \cdot \epsilon(0,\pi,r,\psi)^{-1} \\
&= \epsilon(1/2,\pi,r,\psi) \cdot q_k^{-f(\pi,r,\psi)/2} \cdot \frac{\displaystyle{\lim_{s \rightarrow 0}} ~s L(s,\pi,r)}{L(1,\pi^\vee,r)}, 
\end{align*}
where $f(\pi,r,\psi)$ is the Asai conductor of $\pi$ with respect to $\psi$. It follows from an application of Corollary 7.6 of \cite{akmss18} that the ratio of $L$-values in the above identity is positive. Thus, (\ref{three}) and (\ref{five}) together imply that
\[\lambda \sim \epsilon(1/2,\pi,r,\psi) \cdot \frac{d(\rho)}{d(\pi)} \cdot \ell.\]
Note that the unwritten positive proportionality constant can be made completely explicit from the discussion above and its computation ultimately relies on the validity of the formal degree conjecture for $\GL(n,K)$ and ${\rm U}_n(K/k)$. Also, it is now known that
\[\epsilon(1/2,\pi,r,\psi) = 1,\]
when $\pi$ is a square-integrable representation of $\GL_n(K)$ which is distinguished with respect to $\GL_n(F)$ (cf. \cite[Theorem 1.1]{ana08}, \cite[Theorem 1]{shank18} and \cite[Theorem 3.4.1]{bp18a}, see also \cite[Theorem 1.2]{akmss18} for a more direct approach when $\pi$ is assumed to be cuspidal).

We summarize the above arguments in the following theorem.

\begin{theorem}\label{thm-padic}
Let $\pi$ be a square-integrable representation of $\GL_n(K)$ which is distinguished with respect to $\GL_n(k)$. Let $\rho$ be the square-integrable representation of ${\rm U}(n,K/k)$ that base changes to $\pi$, stably or unstably depending on the parity of $n$. We have
\[\lambda  \sim  \frac{d(\rho)}{d(\pi)} \cdot \ell .\] 
\end{theorem}

\begin{remark}\label{essential}
By \cite[Theorem 1.1]{am17}, we know that the Whittaker function
\[W_\pi = \frac{1}{L(1, \pi_u, r)} W^{\rm ess}_\pi,\]
where $W^{\rm ess}_\pi$ is the essential vector defined by Jacquet, Piatetski-Shapiro, and Shalika, and $\pi_u$ is a certain unramified 
standard module attached to $\pi$, has the property that $\ell(W_\pi) = 1$. We note that $L(1, \pi_u, r) = 1$, except when $\pi$ is a unitary unramified 
twist of the Steinberg representation. Thus, in particular, for $W_\pi = W^{\rm ess}_\pi/ L(1, \pi_u, r)$, we get 
\[\lambda(W_\pi) \sim  d(\rho)/d(\pi).\]
\end{remark}

\begin{remark}\label{fd}
In the finite field $\GL(n)$ setting, notice that the Bessel function indeed satisfies $\ell(B_\pi)=1$ for $\ell$ suitably normalized. 
Hence, Theorem \ref{thm-main}, in the $(\GL_n(E),\GL_n(F))$ case, is the analogue of Remark \ref{essential}, and Theorem \ref{thm-scalar} is the exact analogue of Theorem \ref{thm-padic}. The suppressed constant of Theorem \ref{thm-padic} is also made explicit in Theorem \ref{thm-scalar}. 
\end{remark}

\section*{Acknowledgements}

The initial motivation for this work came from the first named author's joint work with Dipendra Prasad on $\SL(n)$ \cite{ap18}. The authors would like to thank him for several useful comments and clarifications and for his constant encouragement. The authors would like to thank Raphael Beuzart-Plessis, C. S. Rajan, and Vincent S\'echerre for useful conversations. Thanks are also due to the referee for several helpful comments. The second named author would also like to acknowledge the grant ANR-13-BS01-0012 FERPLAY.


\begin{thebibliography}{AKM{+}18}

\bibitem[AKM{+}18]{akmss18}
U.~K. Anandavardhanan, R.~Kurinczuk, N.~Matringe, V.~S\'{e}cherre, and
  S.~Stevens, \emph{Galois self-dual cuspidal types and {A}sai local factors},
  arXiv:1807.07755 (2018), to appear in J. Eur. Math. Soc. 

\bibitem[AKT04]{akt04}
U.~K. Anandavardhanan, Anthony~C. Kable, and R.~Tandon, \emph{Distinguished
  representations and poles of twisted tensor {$L$}-functions}, Proc. Amer.
  Math. Soc. \textbf{132} (2004), no.~10, 2875--2883. \MR{2063106
  (2005g:11080)}

\bibitem[AM17]{am17}
U.~K. Anandavardhanan and Nadir Matringe, \emph{Test vectors for local
  periods}, Forum Math. \textbf{29} (2017), no.~6, 1245--1260. \MR{3719298}

\bibitem[Ana08]{ana08}
U.~K. Anandavardhanan, \emph{Root numbers of {A}sai {$L$}-functions}, Int.
  Math. Res. Not. IMRN (2008), Art. ID rnn125, 25. \MR{2448081}

\bibitem[AP03]{ap03}
U.~K. Anandavardhanan and Dipendra Prasad, \emph{Distinguished representations
  for {${\rm SL}(2)$}}, Math. Res. Lett. \textbf{10} (2003), no.~5-6, 867--878.
  \MR{2025061}

\bibitem[AP18]{ap18}
\bysame, \emph{Distinguished representations for {${\rm SL}(n)$}}, Math. Res.
  Lett. \textbf{25} (2018), no.~6, 1695--1717. \MR{3934841}
  
\bibitem[AR05]{ar05}
U.~K. Anandavardhanan and C.~S. Rajan, \emph{Distinguished representations,
  base change, and reducibility for unitary groups}, Int. Math. Res. Not.
  (2005), no.~14, 841--854. \MR{2146859}

\bibitem[Ber84]{ber84}
Joseph~N. Bernstein, \emph{{$P$}-invariant distributions on {${\rm GL}(N)$} and
  the classification of unitary representations of {${\rm GL}(N)$}
  (non-{A}rchimedean case)}, Lie group representations, {II} ({C}ollege {P}ark,
  {M}d., 1982/1983), Lecture Notes in Math., vol. 1041, Springer, Berlin, 1984,
  pp.~50--102. \MR{748505}

\bibitem[BZ76]{bz76}
I.~N. Bernstein and A.~V. Zelevinski, \emph{Induced representations of the
  group {$GL(n)$} over a {$p$}-adic field}, Funkcional. Anal. i Prilo\v zen.
  \textbf{10} (1976), no.~3, 74--75. \MR{0425031}

\bibitem[BZ77]{bz77}
I.~N. Bernstein and A.~V. Zelevinsky, \emph{Induced representations of
  reductive {$p$}-adic groups. {I}}, Ann. Sci. \'Ecole Norm. Sup. (4)
  \textbf{10} (1977), no.~4, 441--472. \MR{0579172}
  
\bibitem[BP18a]{bp18a}
Rapha\"{e}l Beuzart-Plessis, \emph{Archimedean theory and $\epsilon$-factors for the Asai Rankin-Selberg integrals}, arXiv:1812.00053 (2018).

\bibitem[BP18b]{bp18b}
Rapha\"{e}l Beuzart-Plessis, \emph{Plancherel formula for ${\rm GL}_n(F)\backslash {\rm GL}_n(E)$ and applications to the Ichino-Ikeda and formal degree conjectures for unitary groups}, arXiv:1812.00047 (2018).

\bibitem[Fli88]{fli88}
Yuval~Z. Flicker, \emph{Twisted tensors and {E}uler products}, Bull. Soc. Math.
  France \textbf{116} (1988), no.~3, 295--313. \MR{984899}

\bibitem[Fli92]{fli92}
\bysame, \emph{Distinguished representations and a {F}ourier summation
  formula}, Bull. Soc. Math. France \textbf{120} (1992), no.~4, 413--465.
  \MR{1194271}

\bibitem[Fli93]{fli93}
\bysame, \emph{On zeroes of the twisted tensor {$L$}-function}, Math. Ann.
  \textbf{297} (1993), no.~2, 199--219. \MR{1241802}

\bibitem[Gel70]{gel70}
S.~I. Gelfand, \emph{Representations of the full linear group over a finite
  field}, Mat. Sb. (N.S.) \textbf{83 (125)} (1970), 15--41. \MR{0272916}

\bibitem[GG62]{gg62}
I.~M. Gelfand and M.~I. Graev, \emph{Construction of irreducible
  representations of simple algebraic groups over a finite field}, Dokl. Akad.
  Nauk SSSR \textbf{147} (1962), 529--532. \MR{0148765}

\bibitem[Gow84]{gow84}
Roderick Gow, \emph{Two multiplicity-free permutation representations of the
  general linear group {${\rm GL}(n,q^2)$}}, Math. Z. \textbf{188} (1984),
  no.~1, 45--54. \MR{767361 (86a:20008)}

\bibitem[HII08]{hii08}
Kaoru Hiraga, Atsushi Ichino, and Tamotsu Ikeda, \emph{Formal degrees and
  adjoint {$\gamma$}-factors}, J. Amer. Math. Soc. \textbf{21} (2008), no.~1,
  283--304. \MR{2350057}

\bibitem[Kab04]{kab04}
Anthony~C. Kable, \emph{Asai {$L$}-functions and {J}acquet's conjecture}, Amer.
  J. Math. \textbf{126} (2004), no.~4, 789--820. \MR{2075482}

\bibitem[Kaw77]{kaw77}
Noriaki Kawanaka, \emph{On the irreducible characters of the finite unitary
  groups}, J. Math. Soc. Japan \textbf{29} (1977), no.~3, 425--450.
  \MR{0450383}

\bibitem[Lus00]{lus00}
G.~Lusztig, \emph{{$G(F_q)$}-invariants in irreducible {$G(F_{q^2})$}-modules},
  Represent. Theory \textbf{4} (2000), 446--465. \MR{1780718}

\bibitem[Mao01]{mao01}
Zhengyu Mao, \emph{Spherical {B}essel functions of {${\rm GL}_2({\bf
  F}_{q^2})$}}, J. Number Theory \textbf{87} (2001), no.~1, 154--171.
  \MR{1816041}

\bibitem[Mat14]{mat14}
Nadir Matringe, \emph{Unitary representations of {${\rm GL}(n,K)$}
  distinguished by a {G}alois involution for a {$p$}-adic field {$K$}}, Pacific
  J. Math. \textbf{271} (2014), no.~2, 445--460. \MR{3267536}

\bibitem[Mok15]{mok15}
Chung~Pang Mok, \emph{Endoscopic classification of representations of
  quasi-split unitary groups}, Mem. Amer. Math. Soc. \textbf{235} (2015),
  no.~1108, vi+248. \MR{3338302}

\bibitem[Nie14]{nie14}
Chufeng Nien, \emph{A proof of the finite field analogue of {J}acquet's
  conjecture}, Amer. J. Math. \textbf{136} (2014), no.~3, 653--674.
  \MR{3214273}

\bibitem[Pra99]{pra99}
Dipendra Prasad, \emph{Distinguished representations for quadratic extensions},
  Compositio Math. \textbf{119} (1999), no.~3, 335--345. \MR{1727136
  (2001b:22016)}
  
\bibitem[Pra15]{pra15}
Dipendra Prasad, \emph{A 'relative' local Langlands correspondence}, arXiv:1512.04347 (2015).   

\bibitem[Sha84]{sha84}
Freydoon Shahidi, \emph{Fourier transforms of intertwining operators and
  {P}lancherel measures for {${\rm GL}(n)$}}, Amer. J. Math. \textbf{106}
  (1984), no.~1, 67--111. \MR{729755}
  
  \bibitem[Sha18]{shank18}
D. Shankman, \emph{Local Langlands correspondence for Asai $L$-functions and epsilon factors}, 
arXiv:1810.11852 (2018).

\bibitem[Shi76]{shi76}
Takuro Shintani, \emph{Two remarks on irreducible characters of finite general
  linear groups}, J. Math. Soc. Japan \textbf{28} (1976), no.~2, 396--414.
  \MR{0414730}
  
\bibitem[Yan19]{yang19}
Chang Yang, \emph{Distinguished representations, Shintani base change and a finite field analogue of a conjecture of Prasad}, arXiv:1905.12205 (2019).

\end{thebibliography}
\end{document}